\documentclass[11pt]{article}
\usepackage{mathrsfs,bbm}
\usepackage{slashbox}
\usepackage{nicefrac}
\usepackage[T1]{fontenc}
\usepackage{latexsym,amssymb,amsmath,amsfonts,amsthm}
\usepackage{graphicx}
\usepackage{caption}
\usepackage[utf8]{inputenc}
\usepackage{authblk}
\newcommand{\R}{{\mathbb R}}

\newcommand{\C}{{\mathbb C}}

\newcommand{\be}{\begin{eqnarray}}
\newcommand{\ben}{\begin{eqnarray*}}
\newcommand{\en}{\end{eqnarray}}
\newcommand{\enn}{\end{eqnarray*}}
\newcommand{\ba}{\backslash}
\newcommand{\pa}{\partial}

\newcommand{\ov}{\overline}

\newcommand{\g}{\gamma}

\newcommand{\Om}{\Omega}

\newcommand{\wi}{\widetilde}

\newtheorem{theorem}{Theorem}[section]
\newtheorem{lemma}[theorem]{Lemma}

\begin{document}
\renewcommand{\theequation}{\arabic{section}.\arabic{equation}}
\begin{titlepage}
\title{\bf Recovery of an embedded obstacle and its surrounding medium by formally-determined scattering data}
\author[1]{Hongyu Liu\thanks{hongyuliu@hkbu.edu.hk (HL)}}
\author[2]{Xiaodong Liu\thanks{Corresponding author, xdliu@amt.ac.cn (XL)}}
\affil[1]{Department of Mathematics, Hong Kong Baptist University, Kowloon, Hong Kong SAR, P. R. China.}
\affil[2]{Institute of Applied Mathematics, Academy of Mathematics and Systems Science, Chinese Academy of Sciences, 100190 Beijing, P. R. China.}
\date{}
\end{titlepage}

\maketitle

\begin{abstract}
We consider an inverse acoustic scattering problem in simultaneously recovering an embedded obstacle and its surrounding inhomogeneous medium by formally determined
far-field data. It is shown that the knowledge of the scattering amplitude with a fixed incident direction and all observation angles along with frequencies from an open interval
can be used to uniquely identify the embedded obstacle, sound-soft or sound-hard disregarding the surrounding medium. Furthermore, if the surrounding inhomogeneous medium is from
an admissible class (still general), then the medium can be recovered as well. Our argument is based on deriving certain integral identities involving the unknowns and then
inverting them by certain harmonic analysis techniques.
Finally, based on our theoretical study, a fast and robust sampling method is proposed to reconstruct the shape and location of the buried targets and the support of the surrounding inhomogeneities.

\vspace{.2in}

\end{abstract}

\section{Introduction}\label{sec:Intro}

In this article we are concerned with the inverse scattering problem in recovering unknown/inaccessible objects by acoustic wave probe associated to the Helmholtz system.
It serves as a prototype model to many inverse problems arising from scientific and technological applications \cite{Amm3,CK,Isa,Uhl}.
The unknown/inaccessible object is usually referred to as a scatterer and it could be an impenetrable obstacle or a penetrable inhomogeneous medium.
Many existing studies tend to consider the recovery of either an obstacle or an inhomogeneous medium. We consider the simultaneous recovery of an embedded obstacle and its surrounding inhomogeneous medium, which makes the corresponding study radically challenging.

Mathematically, the inverse scattering problem is described by the following Helmholtz system.
Let $D\subset\mathbb{R}^n$, $n=2,3$, be a bounded Lipschitz domain with a connected complement $D^c:=\mathbb{R}^n\backslash\overline{D}$. Let $\Omega$ be a bounded Lipschitz domain in $\mathbb{R}^n$ such that $D\Subset\Omega$ and $\Omega^c$ is connected. Let $V(x)\in L^\infty(\mathbb{R}^n\backslash\overline{D})$ be a complex-valued function with $\Im V\geq 0$ and $\mathrm{supp}(V)\subset\Omega\backslash\overline{D}$. $D$ and $(\Omega\backslash\overline{D}, V)$, respectively, signify the impenetrable obstacle and the penetrable medium, where $V$ represents the medium parameter. Let $u(x)$ be a complex-valued function that represents the wave pressure. The time-harmonic acoustic scattering is described by the Helmholtz system as follows,
\begin{equation}\label{eq:Helm1}
\Big(\Delta+k^2(1+V)\Big) u(x)=0,\quad x\in\mathbb{R}^n\backslash\overline{D},
\end{equation}
where $u(x)=e^{\mathrm{i}kx\cdot d}+u^s(x)$, $x\in\mathbb{R}^n\backslash\overline{D}$. To complement the Helmholtz system \eqref{eq:Helm1}, we prescribe the following boundary condition on $\partial D$,
\begin{equation}\label{eq:bc}
\mathcal{B}[u](x)=0,
\end{equation}
and the following Sommerfeld radiation condition at $+\infty$,
\begin{equation}\label{eq:sommerfeld}
\lim_{|x|\rightarrow+\infty}|x|^{(n-1)/2}\Big(\frac{\partial u^s}{\partial|x|}-\mathrm{i}k u^s \Big)=0.
\end{equation}
Here, $u^i(x):=e^{\mathrm{i}kx\cdot d}$ with $k\in\mathbb{R}_+$ and $d\in\mathbb{S}^{n-1}:=\{x\in\mathbb{R}^n; |x|=1\}$ is known as the time-harmonic plane wave,
which is an entire solution to $-\Delta w-k^2 w=0$. In \eqref{eq:bc}, $\mathcal{B}[u]=u$ if $D$ is a sound-soft obstacle and $\mathcal{B}[u]=\partial u/\partial\nu$
if $D$ is a sound-hard obstacle, where $\nu\in\mathbb{S}^{n-1}$ denotes the exterior unit normal vector to $\partial D$.
We refer to \cite{KirschLiu,LSSZ} for the unique existence of an $H^1_{loc}(\mathbb{R}^n\backslash\overline{D})$ solution to the Helmholtz system \eqref{eq:Helm1}--\eqref{eq:sommerfeld}.
It is known that $u^s$ has the following asymptotic expansion \cite{CK,Ned},
\begin{equation}\label{us-asymptotic}
u^{s}(x,k,d)=\g_n\frac{e^{ik|x|}}{|x|^{\frac{n-1}{2}}}\left\{u^{\infty}\left(\frac{x}{|x|},k,d\right)+\mathcal{O}\biggl(\frac{1}{|x|}\biggr)\right\}\;
             \mbox{as }\;|x|\rightarrow\infty,
 \end{equation}
that holds uniformly in $\hat{x}:=x/|x|\in \mathbb{S}^{n-1}$, where
\begin{equation}
\g_{n}=\left\{
         \begin{array}{ll}
           \frac{1}{4\pi}, & \hbox{$n=3$;}\medskip \\
           \frac{e^{\mathrm{i}\pi/4}}{\sqrt{8k\pi}}, & \hbox{$n=2$,}
         \end{array}
       \right.
\end{equation}
is a dimensional parameter. $u^\infty(\hat x, k, d)$ is known as the scattering amplitude, where $\hat{x}$, $k$ and $d$ are referred to as the observation angle,
wavenumber and incident direction, respectively.
The inverse scattering problem that we are concerned with is to recover $D$ and $(\Omega\backslash\overline{D}, V)$ by knowledge of $u^\infty(\hat x, k, d)$.
It is noted that $u^\infty(\hat x, k, d)$ is (real) analytic in all of its arguments (cf. \cite{CK,Isa}), and hence if the scattering amplitude is known for $\hat x$
from an open subset of $\mathbb{S}^{n-1}$, then it is known on the whole sphere $\mathbb{S}^{n-1}$. The same remark holds equally for $k$ and $d$.

There is a fertile mathematical theory for the inverse scattering problem described above.
In this work, we shall be mainly concerned with the unique recovery or identifiability issue; that is, given the measurement data, what kind of unknowns that one can recover.
The unique recovery of purely a sound-soft $D$ by knowledge of $u^\infty(\hat x, k, d)$ for either i)~all $\hat x$ and $d$ along with a fixed $k$; or ii)~all $\hat x$ and $k$
along with a fixed $d$; is due to Schiffer's spectral argument \cite{CK,Isa,LR}.
The unique recovery of purely a sound-hard $D$ by knowledge of $u^\infty(\hat x, k, d)$ for all $\hat x$ and $d$ along with a fixed $k$ is due to Isakov's singular source method
\cite{Isa2,KK}. The uniqueness of recovering purely a sound-soft or a sound-hard $D$ by knowledge of $u^\infty(\hat x, k, d)$ with all $\hat x$ and finitely many $k$ and $d$ were
considered in \cite{AR,CY,HNS,LZ}.
The uniqueness of recovery of purely an inhomogeneous medium $(\Omega, V)$ by $u^\infty(\hat x, k, d)$ for all $\hat x$ and $d$ along with a fixed $k$ is mainly due to the
CGO (complex geometrical solutions) approach pioneered by Sylvester and Uhlmann \cite{SU,Na}.
The recovery of a complex scatterer as described above consisting of both an obstacle $D$ and a medium $(\Omega\backslash\overline{D}, V)$ was also considered in the literature,
and the study in this case is also related to the so-called partial data inverse problem \cite{IUY}.
If $u^\infty(\hat x, k, d)$ for all $\hat x$ and $d$ but with a fixed $k$ is used, then the recovery results were obtained by assuming that either $D$ is known in advance
or $(\Omega\backslash\overline{D}, V)$ is known in advance \cite{KP,LZZ,LiuZhangHu,Od,IUY}.
We also refer to \cite{KirschLiu,Liu,LiuZhangsiam} for the reconstruction of the support of $D$ and $\Om$ under certain conditions on $V$.
If $u^\infty(\hat{x}, k, d)$ is known for all $\hat{x}, d$ and $k$, H\"{a}hner \cite{Hahner} show that the simply connected, sound-soft obstacle $D$
together with the surrounding inhomogeneous medium $(\Omega\backslash \overline{D}, V)$ in $\R^2$ can be uniquely determined.
Actually, H\"{a}hner use the limit $k\rightarrow 0$ and obtain uniqueness of the obtacle $D$ from Schiffer's uniqueness result. In the arguments, only a single
incident direction $d$ is used and the result can be extended to obstacles with several connected components. However, uniquness of the surrounding inhomogeneity need
all the incident directions.
To our best knowledge, there is no unique recovery result available in the literature in recovering both $D$ and $(\Omega\backslash\overline{D}, V)$ by knowledge of
$u^\infty(\hat{x}, k, d)$ for both $\hat x$ and $k$, but a fixed $d$.
It is noted that the inverse scattering problem is formally determined with the data just mentioned, and we shall consider it in the present article.
Finally, we would like to mention that there are some other studies by making use of dynamical measurement data in recovering an inhomogeneous medium \cite{KT}.

For the proposed inverse scattering problem, our mathematical argument can be briefly sketched as follows.
First, we derive the integral representation of the solution to the scattering problem involving both the obstacle $D$ and the medium $(\Omega\backslash\overline{D}, V)$.
Then, by considering the low wavenumber asymptotics in terms of $k$, we can derive certain integral identities,
which can serve to decouple the scattering information of $D$ from that of $(\Omega\backslash\overline{D}, V)$.
Finally, by using certain harmonic analysis techniques, we can invert the previously obtained integral identities to recover the obstacle and the medium.
Inspired by the theoretical study in the current article as well as a recent work \cite{Liu2016} by one of the authors, where a fast and robust direct sampling method is proposed using the far-field patterns $u^\infty(\hat x, k, d)$
with many $\hat{x},d\in S^{n-1}$ and a single $k$, we develop a similar method by using the far-field patterns $u^\infty(\hat x, k, d)$
with many $\hat{x}\in S^{n-1}$, many $k$ in an interval and one or few $d\in S^{n-1}$, to reconstruct the shape and location of the buried target and the support of the surrounding inhomogeneity.

The rest of the paper is organised as follows. In Section 2, we present some preliminary knowledge on the boundary layer potentials and volume potentials.
Section 3 is devoted to the derivation of the integral representation of the forward scattering problem. In Section 4, we present the simultaneous recovery results.
Finally, in Section 5, a sampling method based on the idea from the uniqueness analyses is proposed to reconstruct the support of the buried object and the surrounding inhomogeneity.

\section{Preliminaries on integral operators}

Let $B_R$ be a central ball of radius $R$ such that $\Omega\Subset B_R$. Set $D^R:=D^c\cap B_R$. Let $\Phi(x, y)$, $x, y\in\mathbb{R}^n$ and $x\neq y$, be the fundamental solution
to the Helmholtz equation, given by
\be\label{Phi}
\Phi(x,y):=\left\{
              \begin{array}{ll}
                \frac{\mathrm{i}k}{4\pi}h^{(1)}_0(k|x-y|)=\frac{e^{\mathrm{i}k|x-y|}}{4\pi|x-y|}, & n=3;\medskip \\
                \frac{\mathrm{i}}{4}H^{(1)}_0(k|x-y|), & n=2,
              \end{array}
            \right.
\en
where $h^{(1)}_0$ and $H^{(1)}_0$ are, respectively, the spherical Hankel function and Hankel function of the first kind and order zero.
For any $\varphi\in H^{-1/2}(\pa D)$, $\psi\in H^{1/2}(\pa D)$ and $\phi\in L^2(\Omega\backslash\overline{D})$,
the single-layer potential is defined by
\ben
(\mathcal{S}\varphi)(x):=\int_{\pa D}\varphi(y)\Phi(x,y)ds(y),\quad x\in\R^n\ba{\pa D},
\enn
the double-layer potential is defined by
\ben
(\mathcal{K}\psi)(x):=\int_{\pa D}\psi(y)\frac{\pa\Phi(x,y)}{\pa\nu(y)}ds(y),\quad x\in\R^n\ba{\pa D},
\enn
and the volume potential is defined by
\ben
(\mathcal{G}_V\phi)(x):=\int_{\Omega\backslash\overline{D}}\Phi(x,y)V(y)\phi(y)dy\quad \mbox{for\ $x\in \R^{n}$},
\enn
respectively. It is shown in \cite{Mclean} that the potentials $\mathcal{S}: H^{-1/2}(\pa D)\rightarrow H^{1}_{loc}(\R^n\ba\pa D)$,
$\mathcal{K}: H^{1/2}(\pa D)\rightarrow H^{1}_{loc}(\R^n\ba\ov{D})$, $\mathcal{K}: H^{1/2}(\pa D)\rightarrow H^{1}(D)$ and
$\mathcal{G}_V: L^2(\Omega\backslash\overline{D})\rightarrow H^2(D^R)$ are well defined.
We also define the restriction of $\mathcal{S}$ and $\mathcal{K}$ to the  boundary $\pa D$ by
 \be
\label{S} (S\varphi)(x):= \int_{\pa D}\Phi(x,y)\varphi(y)ds(y),\quad  x\in \pa D,\\
\label{K} (K\psi)(x):= \int_{\pa D}\frac{\pa\Phi(x,y)}{\pa\nu(y)}\psi(y)ds(y),\quad  x\in \pa D
 \en
and the restriction of the normal derivative of $\mathcal{S}$ and $\mathcal{K}$ to the  boundary $\pa D$ by
 \be
\label{Kp} (K^{'}\varphi)(x):=\frac{\pa}{\pa\nu(x)}\int_{\pa D}\Phi(x,y)\varphi(y)ds(y),\quad  x\in \pa D,\\
\label{T} (T\psi)(x):=\frac{\pa}{\pa\nu(x)}\int_{\pa D}\frac{\pa\Phi(x,y)}{\pa\nu(y)}\psi(y)ds(y),\quad  x\in \pa D.
 \en
These boundary operators $S: H^{-1/2}(\pa D)\rightarrow H^{1/2}(\pa D)$, $K: H^{1/2}(\pa D)\rightarrow H^{1/2}(\pa D)$,
$K^{'}: H^{-1/2}(\pa D)\rightarrow H^{-1/2}(\pa D)$ and $T: H^{1/2}(\pa D)\rightarrow H^{-1/2}(\pa D)$ are well defined \cite{Mclean,Ned}.
The restriction of the volume potential $\mathcal{G}_V\phi$ on the boundary $\pa D$ is signified by $G_{V}\phi$, the corresponding normal derivative is denoted by $\pa_{\nu}G_{V}\phi$.

\section{Integral representation for forward scattering problem}\label{sec3}

For the subsequent use of our studying the inverse problem, we derive in this section a certain new integral representation of the solution to the forward scattering problem
(\ref{eq:Helm1})-(\ref{eq:sommerfeld}).

\begin{theorem}\label{uni.direct}
The forward scattering problem \eqref{eq:Helm1}-\eqref{eq:sommerfeld} has at most one solution.
\end{theorem}
\begin{proof}
Clearly, it is sufficient to show that $u=0$ in $\R^{n}\ba\ov{D}$ if $u^{i}=0$ in $\R^n$.
Using Green's theorem in $B_R\ba\ov{D}$, with the aid of (\ref{eq:Helm1}) and \eqref{eq:bc}, we obtain that
\ben
 \int_{\pa B_{R}}u\frac{\pa\ov{u}}{\pa\nu}ds=\int_{B_{R}\ba\ov{D}} \Big[|\nabla u|^{2}-k^2(1+\ov{V})|u|^2 \Big]\ dx.
\enn
From this, since $k>0$ and $\Im V\geq 0$, it follows that
 \ben
  \Im\int_{\pa B_{R}}u\frac{\pa \ov{u}}{\pa\nu}ds =k^2\int_{B_{R}\ba\ov{D}}\Im V |u|^2dx\geq0.
 \enn
By Rellich's Lemma (cf. \cite{CK}), we deduce that $u=0$ in $\R^{n}\ba B_{R}$ and it follows by the unique continuation principle that $u=0$ in $\R^{n}\ba\ov{D}$.

The proof is complete.
\end{proof}

Now, we turn to the existence of the solution to the forward scattering problem \eqref{eq:Helm1}-\eqref{eq:sommerfeld} via the integral equation method.

\begin{theorem}\label{Integralrepresentations}
Let $u^{i}=e^{{\rm i}kx\cdot d}$ be an incident plane wave with the wavenumber $k>0$ and incident direction $d\in \mathbb{S}^{n-1}$, and consider the scattering problem
\eqref{eq:Helm1}-\eqref{eq:sommerfeld}. Let $u\in H^1_{loc}(\R^n\ba\ov{D})$ be a solution to the scattering problem \eqref{eq:Helm1}-\eqref{eq:sommerfeld}.

{ (i)}~ Assume that $D$ is sound-soft, then the total wave field $u|_{D^R}\in L^2(D^R)$ has the following form
\be\label{urepresentationD}
u=u^{i}+k^2\mathcal{G}_{V}u+ (\mathcal{K}-\mathrm{i}\mathcal{S})\psi \quad{\rm in}\ D^R,
\en
where
$\psi\in H^{1/2}(\pa D)$ is determined by the following boundary integral equation
\be\label{dbcgamma1D}
0 = u^{i}+k^2 G_{V}u+ \frac{1}{2}\psi+(K-\mathrm{i}S)\psi \quad{\rm on}\ \pa D.
\en
Furthermore, for
$(u|_{D^R},\psi)\in L^2(D^R)\times H^{1/2}(\pa D)$,
the system of integral equations \eqref{urepresentationD}--\eqref{dbcgamma1D} is uniquely solvable.

{ (ii)}~Assume that $D$ is sound-hard, then the total wave field $u|_{D^R}\in L^2(D^R)$ has the following form
\be\label{urepresentationN}
u=u^{i}+k^2\mathcal{G}_{V}u+(\mathcal{S}+\mathrm{i}k^3\mathcal{K}\circ\mathscr{S}^2)\varphi \quad{\rm in}\,D_R,
\en
where $\mathscr{S}$ is the single-layer operator defined in \eqref{S} with the wavenumber formally replaced by $k=\mathrm{i}$ and
$\varphi\in H^{-1/2}(\pa D)$ is determined by the following boundary integral equation
\be\label{ibcgamma2N}
0 = \frac{\pa (u^{i}+k^2G_{V}u)}{\pa\nu}-\frac{1}{2}\varphi+(K^{\prime}+\mathrm{i}k^3T\circ\mathscr{S}^2)\varphi \quad{\rm on}\ \pa D.
\en
Furthermore, for
$(u|_{D^R},\varphi)\in L^2(D^R)\times H^{-1/2}(\pa D),$
the system of integral equations \eqref{urepresentationN}--\eqref{ibcgamma2N} is uniquely solvable.
\end{theorem}

\begin{proof}
We shall only prove the case (i), and the other case (ii) can be shown by following a similar argument.

Let $u|_{D^R}\in L^2(D^R)$ be a solution to the system (\ref{urepresentationD})-(\ref{dbcgamma1D}).
Extending $u$ into $\R^n\ba\ov{D}$ by the right hand side of (\ref{urepresentationD}). By
the mapping properties of the volume and boundary layer potentials (cf. \cite{CK,Mclean}), we have $u\in H^1_{loc}(\R^n\ba\ov{D})$.
Since $\Phi$ is the fundamental solution to the Helmholtz equation, one can deduce
that $\Delta u+k^{2}(1+V)u=0$ in $\R^n\ba\ov{D}$; that is, the equation (\ref{eq:Helm1}) holds.
The boundary condition \eqref{eq:bc} satisfied by $u$ are easily verified by combing the jump relations of layer potentials (cf. \cite{CK,Mclean}) and
the boundary equation (\ref{dbcgamma1D}).
Furthermore, the scattered field $u-u^i$ satisfies the radiation condition \eqref{eq:sommerfeld} due to the fact that the fundamental solution $\Phi$ is radiating.

Next, we show that the system (\ref{urepresentationD})-(\ref{dbcgamma1D}) is uniquely solvable for $(u|_{D^R},\psi)\in L^2(D^R)\times H^{1/2}(\pa D)$.
We write the system (\ref{urepresentationD})-(\ref{dbcgamma1D}) into a matrix form
\be\label{ABF}
(\mathscr {A}+\mathscr{B})X=F
\en
with
\ben
&&\mathscr{A}:=\left(
  \begin{array}{cc}
    -1 & 0 \\
    0 & 1/2+\wi{K}-\mathrm{i}\wi{S} \\
  \end{array}
\right),\quad
X:=\left(
               \begin{array}{c}
                 u \\
                 \psi \\
               \end{array}
             \right),\quad
F:=\left(
                       \begin{array}{c}
                         -u^i \\
                         -u^i \\
                       \end{array}
                     \right),\\
&&\mathscr{B}:=\left(
                       \begin{array}{cc}
                         k^2\mathcal{G}_{V} & \mathcal{K}-\mathrm{i}\mathcal{S} \\
                         k^2G_{V} & (K-\wi{K})-\mathrm{i}(S-\wi{S}) \\
                       \end{array}
                     \right),
\enn
where $\wi{K}$ and $\wi{S}$ are the corresponding operators of $K$ and $S$, respectively, with $\Phi$ replaced by $\Phi_0$ defined in \eqref{Phi0}.
We study the system in $L^2(D^R)\times H^{1/2}(\pa D)$ with respect to the canonical norm. Clearly, the operator $\mathcal {A}$ is a bounded operator that has a bounded
inverse. Furthermore, all entries of the matrix operator $\mathscr{B}$ are compact in the corresponding spaces. This implies that $\mathscr{A}+\mathscr{B}$ is a Fredholm operator.
Thus it suffices for us to show the uniqueness of the system (\ref{ABF}) in $L^2(D^R)\times H^{1/2}(\pa D)$.
Let $F=0$, then Theorem \ref{uni.direct} implies that $u=0$ in $\R^n\ba\ov{D}$.
Define
\ben
v:=(\mathcal{K}-\mathrm{i}\mathcal{S})\psi\quad{\rm in}\, \R^n\ba\pa D.
\enn
Then, $v=u=0$ in $\R^n\ba\ov{D}$ and $v|_{D}\in H^1(D)$ solves the following PDE,
\ben
\Delta v+k^{2}v=0\quad{\rm in}\ \ D.
\enn
Furthermore, the jump relations yield that
 \ben
 -v_{-}=\psi, \ \ -\frac{\pa v_{-}}{\pa\nu}=\mathrm{i}\psi \qquad \mbox{on}\,\pa D,
 \enn
 where $\pm$ signify the approaching of $\pa D$ from inside and outside of $D$. Interchanging the order of integration and using Green's first theorem over $D$, we obtain
 \ben
\mathrm{i}\int_{\pa D}|\psi|^{2}\ ds=\int_{\pa D}\ov{v_{-}}\frac{\pa v_{-}}{\pa\nu}\ ds=\int_{D}{|\nabla v|^{2}-k^{2}|v|^{2}}\ dx.
 \enn
Taking the imaginary parts of both sides of the above equation readily yields that $\psi=0$ on $\pa D$.

The proof is complete.
\end{proof}

In Theorem \ref{Integralrepresentations}, we choose an approach in using a combined form of volume, double- and single-layer potential.
Such a combination makes the integral equations uniquely solvable for all wavenumber.
However, difficulties will arise in the
study of the low wavenumber behavior of solutions to the exterior Dirichlet problems for the Helmholtz equation in two dimensions, where the fundamental solution
$H^{(1)}_0(k|x-y|)$ in the single-layer potential has no limit as $k\rightarrow 0$.
Nevertheless, by following the idea in \cite{Kress2D} due to Kress, and using a similar argument as in the proof of Thoerem \ref{Integralrepresentations},
we can obtain the following solution representation \eqref{urepresentationD2}-\eqref{dbcgamma1D2} for the exterior Dirichlet problem in the two dimensional case.

\begin{theorem}\label{thm:3}
Assume that $D\subset \R^2$ is sound-soft, let $u\in H^1_{loc}(\R^2\ba\ov{D})$ be a solution to the scattering problem \eqref{eq:Helm1}-\eqref{eq:sommerfeld}.
Then the total field $u|_{D^R}\in L^2(D^R)$ can also be given in the following form
\be\label{urepresentationD2}
u=u^{i}+k^2\mathcal{G}_{V}u+ \Big[\mathcal{K}+\mathcal{S}\circ\Big(W-\frac{2\pi}{\ln k}\Big)\Big]\psi \quad{\rm in}\,D_R,
\en
where $W: H^{-1/2}(\pa D)\rightarrow H^{-1/2}(\pa D)$ is defined by \eqref{W} and
$\psi\in H^{1/2}(\pa D)$ is determined by the following boundary integral equation
\be\label{dbcgamma1D2}
0 = u^{i}+k^2 G_{V}u+ \frac{1}{2}\psi+\Big[K+S\circ\Big(W-\frac{2\pi}{\ln k}\Big)\Big]\psi \quad{\rm on}\ \pa D.
\en
Furthermore, for
$(u|_{D^R},\psi)\in L^2(D^R)\times H^{1/2}(\pa D)$,
the system of integral equations \eqref{urepresentationD2}--\eqref{dbcgamma1D2} is uniquely solved.

\end{theorem}

\section{Unique recovery results}

In this section, we are in a position to present the major recovery results for the proposed inverse problem in determining $D$ and $(\Omega\backslash\overline{D}, V)$ by knowledge of $u^\infty(\hat x, k, d)$ for all $\hat x$, $k$ but a fixed $d$.

\subsection{Low-wavenumber asymptotics}

For the subsequent use, we first derive the low-wavenumber asymptotic expansions of the integral representations of solutions in Theorem~\ref{Integralrepresentations}.
Recall that the fundamental solution in $\R^n$ of the Laplace's equation is given by
\be\label{Phi0}
\Phi_0(x,y):=\left\{
               \begin{array}{ll}
                \displaystyle{ \frac{1}{4\pi |x-y|}}, & \hbox{in $\R^3$;}\medskip \\
                 \displaystyle{\frac{1}{2\pi}\ln\frac{1}{|x-y|}}, & \hbox{in $\R^2$.}
               \end{array}
             \right.
\en
In what follows, for a potential operator introduced in the previous section, say $\mathcal{S}$, we use $\wi{\mathcal{S}}$ to denote the corresponding integral operator with $\Phi$ replaced by $\Phi_0$ defined in \eqref{Phi0}.
Using the series expansion for $e^{\mathrm{i}k|x-y|}$ of $\Phi$ in $\R^3$ and the expansions of the Bessel function $J_0$ and the Neumann function $Y_0$ of order $0$ (see Section 3.4 in \cite{CK})
in $\R^2$, respectively, the fundamental solution to the Helmholtz equation has the following asymptotic expansion as $k\rightarrow +0$,
\be\label{Phiasy}
\begin{split}
& \Phi(x,y)= \Phi_0(x,y)\\
&+\left\{
               \begin{array}{ll}
                 \frac{\mathrm{i}}{4\pi}k-\frac{|x-y|}{8\pi}k^2+\mathcal{O}(k^3), & \hbox{in $\R^3$;} \medskip\\
                 -\frac{1}{2\pi}\ln k +c_2+ \frac{|x-y|}{8\pi}k^2\ln k+\Psi(x,y)k^2+ \mathcal{O}(k^4\ln k), & \hbox{in $\R^2$,}
               \end{array}
             \right.
\end{split}
\en
where $c_2:={\ln2}/{2\pi}-{C}/{2\pi}+{\mathrm{i}}/{4}$ with $C=0.5722\ldots$ denoting the Euler constant and
\ben
\Psi(x,y):=\frac{|x-y|^2}{8\pi}\left(\ln\frac{|x-y|}{2}+C-1-\frac{\mathrm{i}\pi}{2}\right).
\enn
For $\varphi\in H^{-1/2}(\pa D)$, we define the operators $L: H^{-1/2}(\pa D)\rightarrow \C$, $M, N: H^{-1/2}(\pa D)\rightarrow H^{1/2}(\pa D)$
and $W: H^{-1/2}(\pa D)\rightarrow H^{-1/2}(\pa D)$, respectively, by
\be
\nonumber L\varphi&:=&\int_{\pa D}\varphi(y)\ ds(y);\\
\nonumber (M\varphi)(x)&:=&\int_{\pa D}\frac{|x-y|}{8\pi}\varphi(y)\ ds(y),\quad x\in\pa D;\\
\nonumber (N\varphi)(x)&:=&\int_{\pa D}\Psi(x,y)\varphi(y)\ ds(y),\quad x\in\pa D;\\
\label{W} (W\varphi)(x)&:=&\varphi-\frac{1}{|\pa D|}\int_{\pa D}\varphi(y)\ ds(y),\quad x\in\pa D.
\en
It is clear that $L\circ W\equiv 0$.
Denote by $\mathcal {M}\varphi$ and $\mathcal {N}\varphi$ the potentials in $\R^n\ba\pa D$ by the right hand sides of $M\varphi$ and $N\varphi$, respectively.
For $\psi \in H^{1/2}(\pa D)$, we also introduce the operators $P, P^{\prime}, Q, Q^{\prime}: H^{1/2}(\pa D)\rightarrow H^{1/2}(\pa D)$ as follows,
\ben
(P\psi)(x)&:=&\frac{1}{8\pi}\int_{\pa D}\frac{\pa |x-y|}{\pa\nu(y)}\  \psi(y)ds(y),\quad x\in \pa D,\\
(P^{\prime}\psi)(x)&:=&\frac{1}{8\pi}\int_{\pa D}\frac{\pa |x-y|}{\pa\nu(x)}\psi(y)\ ds(y),\quad x\in \pa D,\\
(Q\psi)(x)&:=&\int_{\pa D}\frac{\pa \Psi(x,y)}{\pa\nu(y)}\psi(y)\ ds(y),\quad x\in \pa D,\\
(Q^{\prime}\psi)(x)&:=&\int_{\pa D}\frac{\pa \Psi(x,y)}{\pa\nu(x)}\psi(y)\ ds(y),\quad x\in \pa D.
\enn
Denote by $\mathcal {P}\psi$ and $\mathcal {Q}\psi$ the potentials in $\R^n\ba\pa D$ by the right hand sides of $P\psi$ and $Q\psi$, respectively.
Finally, for $\phi\in L^2(\Om\ba\ov{D})$, we define
\ben
U_V\phi:=\int_{\Omega\backslash\overline{D}}V(y)\phi(y)\ dy.
\enn

\begin{lemma}Let $u^{i}=e^{\mathrm{i}kx\cdot d}$ be an incident plane wave with the wavenumber $k>0$ and incident direction $d\in \mathbb{S}^{n-1}$,
and consider the scattering problem \eqref{eq:Helm1}-\eqref{eq:sommerfeld}.

{ (i)}~Assume that $D$ is sound-soft.
In the two dimensional case, the total wave field has the following asymptotic expansion
\be\label{uasyD2D}
u
&=&\mathcal {F}(1)+\sum^{\infty}_{m=1}C_m(\mathcal {F}, \wi{S}, L, A)\left(\frac{1}{\ln k}\right)^m+\mathrm{i}k\mathcal {F}(x\cdot d)\cr
& &+\sum^{\infty}_{m=1}D_m(\mathcal {F}, \wi{S}, L, A)k\left(\frac{1}{\ln k}\right)^m\cr
& &-k^2\ln k\Big[\frac{1}{2\pi}\mathcal {F}\circ U_V\circ\mathcal {F}(1)+\mathcal {F}\circ(P+M\circ W)\circ A(1)\Big]\cr
& &-k^2\Big[\frac{1}{2}\mathcal {F}(x\cdot d)^2-\mathcal {F}\circ(1+c_2)U_V\circ\mathcal {F}(1)-\mathcal {F}\circ\wi{G_V}\circ\mathcal {F}(1)\cr
& &\qquad -\mathcal {F}\circ U_V\circ\mathcal {F}\circ (\wi{S}+c_2L)\circ A(1)\cr
& &\qquad -2\pi\mathcal {F}\circ (P+M\circ W)\circ [A\circ(\wi{S}+c_2L)-I]\circ A(1)\cr
& &\qquad -\mathcal {F}\circ(-Q+2\pi M-N\circ W)\circ A(1)\Big]\cr
& &+\mathcal{O}\left(\frac{k^2}{\ln k}\right)\quad {\rm as}\ \ k\rightarrow +0\quad {\rm in}\ \ \R^2\ba\ov{D},
\en
where $A:=(I/2+\wi{K}+L+\wi{S}\circ W)^{-1}$, $\mathcal {F}:=I-(\wi{\mathcal{K}}+L+\wi{\mathcal{S}}\circ W)\circ A$, $C_m$ and $D_m$ are functionals defined by the operators
$\mathcal {F}, \wi{S}, L, A$.
In the three dimensional case, the total wave field has the following asymptotic expansion
\be\label{uasyD3D}
u
&=&\mathcal {F}(1)+\mathrm{i}k\mathcal {F}(x\cdot d)\cr
& & +k^2\Big[-\frac{1}{2}\mathcal {F}(x\cdot d)^2+\mathcal {F}\circ \wi{G}_V\circ\mathcal {F}(1)-\mathcal {F}\circ (P-iM)\circ A(1)\cr
& &\qquad +\frac{1}{4\pi} \mathcal {F}\circ L \circ A \circ L \circ A(1)-\frac{i}{4\pi}\mathcal {F}\circ L \circ A (x\cdot d))\Big]\cr
& & +\mathcal{O}(k^3)\quad {\rm as}\ \ k\rightarrow +0\quad {\rm in}\ \ \R^3\ba\ov{D},
\en
where $A:=(I/2+\wi{K}-i\wi{S})^{-1}$ and $\mathcal {F}:=I-(\wi{\mathcal{K}}-i\wi{\mathcal{S}})\circ A$.

{ (ii)}~Assume that $D$ is sound-hard.
In the two dimensional case, the total wave field has the following asymptotic expansion
\be\label{uasyN2D}
u
&=&1+\mathrm{i}k\Big[x\cdot d+\wi{\mathcal{S}}\circ B(d\cdot\nu)\Big]+\frac{k^2\ln k}{2\pi}U_V(1)\cr
& &+k^2\Big[-\frac{(x\cdot d)^2}{2}-\wi{\mathcal{S}}\circ B(d\cdot\nu)(x\cdot d)+\wi{\mathcal{G}}_V(1)+c_2U_V(1)\cr
&&\qquad+\wi{\mathcal{S}}\circ B\circ\frac{\pa \wi{G}_V(1)}{\pa\nu}\Big]\cr
& &+\mathcal{O}(k^3\ln k)\quad {\rm as}\ \ k\rightarrow 0\quad {\rm in}\ \ \R^2\ba\ov{D},
\en
where $B:=(I/2-\wi{K}^{\prime})^{-1}$.
In the three dimensional case, the total wave field has the following asymptotic expansion
\be\label{uasyN3D}
u
&=&1+\mathrm{i}k\left[x\cdot d+\wi{\mathcal{S}}\circ B(d\cdot\nu)\right]\cr
& & + k^2\Big[-\frac{(x\cdot d)^2}{2}-\wi{\mathcal{S}}\circ B(x\cdot d)(\nu\cdot d)-\frac{1}{4\pi}L\circ B (\nu\cdot d)\cr
& & \qquad +\wi{\mathcal{G}}_{V}(1)+\wi{\mathcal{S}}\circ B\circ \frac{\pa \wi{G}_V(1)}{\pa\nu}\Big]\cr
& &+\mathcal{O}(k^3)\quad {\rm as}\ \ k\rightarrow 0\quad {\rm in}\ \ \R^3\ba\ov{D},
\en
where $B:=(I/2-\wi{K}^{\prime})^{-1}$.
\end{lemma}

\begin{proof}
We present the proof of \eqref{uasyD2D} only, and the other asymptotic expansions can be proved by following a similar argument.

Rewrite \eqref{urepresentationD2}-\eqref{dbcgamma1D2} into a matrix form
\ben
\left(
  \begin{array}{cc}
    I-k^2\mathcal {G}_{V} & -\mathcal {K}-\mathcal{S}\circ\Big(W-\frac{2\pi}{\ln k}\Big) \\
    k^2 G_{V} & \frac{1}{2}I+K+S\circ\Big(W-\frac{2\pi}{\ln k}\Big) \\
  \end{array}
\right)\left(
         \begin{array}{c}
           u \\
           \psi \\
         \end{array}
       \right)=\left(
                 \begin{array}{c}
                   u^{i} \\
                   -u^{i} \\
                 \end{array}
               \right).
\enn
We can deduce that
\be\label{upsimatrix}
\left(
         \begin{array}{c}
           u \\
           \psi \\
         \end{array}
       \right)=
\left(
  \begin{array}{cc}
    I-k^2\mathcal {G}_{V} & -\mathcal {K}-\mathcal{S}\circ\Big(W-\frac{2\pi}{\ln k}\Big) \\
    k^2 G_{V} & \frac{1}{2}I+K+S\circ\Big(W-\frac{2\pi}{\ln k}\Big) \\
  \end{array}
\right)^{-1}\left(
                 \begin{array}{c}
                   u^{i} \\
                   -u^{i} \\
                 \end{array}
               \right).
\en
Recall the asymptotic behavior \eqref{Phiasy}, that is, for $x\neq y$, as $k\rightarrow +0$,
\ben
\Phi(x,y)
&=&-\frac{\ln k}{2\pi}+[\Phi_0(x,y)+c_2]\cr
& &+\frac{|x-y|}{8\pi}k^2\ln k +\Psi(x,y)k^2+\mathcal{O}(k^4\ln k).
\enn
From this, patient but still straightforward, calculations show that, as $k\rightarrow +0$,
\ben
k^2\mathcal {G}_{V}&=&-\frac{k^2\ln k}{2\pi}U_{V}+k^2(\wi{\mathcal {G}_{V}}+c_2U_V)+O(k^2/\ln k),\\
\mathcal {K}&=&\wi{\mathcal {K}}+k^2\ln k\mathcal {P}+k^2\mathcal {Q}+O(k^4\ln k),\\
\mathcal{S}\circ\Big(W-\frac{2\pi}{\ln k}\Big)
&=& L + \wi{\mathcal{S}}\circ W - \frac{2\pi}{\ln k}(\wi{\mathcal{S}}+c_2L) + k^2\ln kM\circ W \cr
& & + k^2(N\circ W-2\pi M) - \frac{2\pi k^2}{\ln k}N + \mathcal{O}(k^4\ln k).
\enn
Inserting these expansions into \eqref{upsimatrix}, using the fact that
\ben
\left(
  \begin{array}{cc}
    I & -(\wi{\mathcal {K}}+L+\wi{\mathcal {S}}\circ W) \\
    0 & I/2 +\wi{K}+L+\wi{S}\circ W \\
  \end{array}
\right)^{-1}=\left(
               \begin{array}{cc}
                 I & (\wi{\mathcal {K}}+L+\wi{\mathcal {S}}\circ W)\circ A \\
                 0 & A \\
               \end{array}
             \right),
\enn
where $A:=(I/2 +\wi{K}+L+\wi{S}\circ W)^{-1}$, we deduce that
\ben
\left(
  \begin{array}{c}
    u \\
    \psi \\
  \end{array}
\right)
&=&
\Big[I+\frac{2\pi}{\ln k}\left(
                           \begin{array}{cc}
                             0 & \mathcal {F}\circ (\wi{S}+c_2L)\\
                             0 & -A\circ (\wi{S}+c_2L)\\
                           \end{array}
                         \right)\cr
& &\, +k^2\ln k\left(
               \begin{array}{cc}
                 \frac{1}{2\pi}\mathcal {F}\circ U_V & -\mathcal {F}\circ(P+M\circ W) \\
                 -\frac{1}{2\pi}A\circ U_V & A\circ(P+M\circ W) \\
               \end{array}
             \right)\cr
& &\, +k^2\left(
          \begin{array}{cc}
            -\mathcal {F}\circ (\wi{G_V}+c_2 U_V) & -\mathcal {F}\circ (Q-2\pi M+N\circ W)\\
            A\circ(\wi{G_V}+c_2 U_V) & A\circ (Q-2\pi M+N\circ W)\\
          \end{array}
        \right)\Big]^{-1}\cr
& &\left(
               \begin{array}{cc}
                 I & (\wi{\mathcal {K}}+L+\wi{\mathcal {S}}\circ W)\circ A \\
                 0 & A \\
               \end{array}
             \right)\left(
                      \begin{array}{c}
                        u^{i} \\
                        -u^{i} \\
                      \end{array}
                    \right).
\enn
Finally, \eqref{uasyD2D} follows by a Neumann series argument, along with the use of the following expansion
\ben
u^{i}(x,k,d)=e^{\mathrm{i}kx\cdot d}=1+\mathrm{i}k(x\cdot d)-k^2\frac{(x\cdot d)^2}{2}+O(k^3)\quad \mbox{as}\,k\rightarrow +0,
\enn

The proof is complete.
\end{proof}

\subsection{Recovery of the embedded obstacle}

We first consider the unique recovery of the embedded obstacle $D$, disregarding the surrounding inhomogeneous medium $(\Omega\backslash\overline{D}, V)$.
To that end, in what follows, we introduce another scatter consisting of an obstacle $\widehat{D}$ and a medium $(\widehat{\Omega}\backslash\overline{\widehat{D}}, \widehat{V})$.
Without loss of generality, we assume that $R$ is large enough such that both $\Omega$ and $\widehat{\Omega}$ are contained in $B_R$.
Throughout the rest of the section, we use $\widehat{u}$ to denote the wave field associated with $\widehat{D}$ and $(\widehat{\Omega}\backslash\overline{\widehat{D}}, \widehat{V})$. In what follows, we shall show that if $u^\infty$ and $\widehat{u}^\infty$ are identically the same for certain measurement data set, then $D$ and $\widehat{D}$ must be identically the same as well disregarding $(\Omega\backslash\overline{D}, V)$ and $(\widehat{\Omega}\backslash\overline{\widehat{D}}, \widehat{V})$. This is always true for the sound-soft case, whereas for the sound-hard case, we need impose a certain generic geometric condition on the obstacles $D$ and $\widehat{D}$ as follows.

Suppose that $D$ and $\widehat{D}$ are both sound-hard and $D\neq\widehat{D}$. Let $\mathbb{G}$ be the unbounded connected component of the complement of $D\cup\widehat{D}$.
If $D\neq\widehat{D}$, we know that either $(\R^2\ba\mathbb{G})\ba\ov{D}$ or $(\R^2\ba\mathbb{G})\ba\ov{\widehat{D}}$ is nonempty.
$D$ and $\widehat{D}$ are said to be {\it admissible} if there exists a connected component, say $D^*$, of $(\R^2\ba\mathbb{G})\ba\ov{D}$ or $(\R^2\ba\mathbb{G})\ba\ov{\widehat{D}}$
such that the divergence theorem holds in $D^*$. Here, we note that divergence theorem always holds in Lipschitz domains (cf. \cite{Mclean}).
It is easily seen that $\partial D^*$ is composed of finitely many Lipschitz pieces.
One can show that if $D^*$ can be decomposed into the union of finitely many Lipschitz subdomains,
then the divergence theorem holds in $D^*$ and hence both $D$ and $\widehat{D}$ are admissible.
Moreover, if both $D$ and $\widehat{D}$ are polyhedral domains, then $D^*$ is also a polyhedral domain and therefore both $D$ and $\widehat{D}$ are clearly admissible.

\begin{theorem}\label{thm:main1}
Let $D$ and $\widehat{D}$ be two obstacles such that $u^{\infty}(\hat{x},k,d)=\widehat{u}^{\infty}(\hat{x},k,d)$ for $(\hat x, k)\in \Sigma\times \zeta$
and a fixed $d\in\mathbb{S}^{n-1}$, where $\Sigma\times\zeta$ is any open subset of $\mathbb{S}^{n-1}\times\mathbb{R}_+$.
Then $D=\widehat{D}$ if they are one of the following two types:
\begin{enumerate}
  \item[(i)] both $D$ and $\widehat{D}$ are sound soft;
  \item[(ii)] both $D$ and $\widehat{D}$ are sound hard and satisfy the admissibility condition as described above.
\end{enumerate}
\end{theorem}

\begin{proof}
Assume by contradiction that $D\neq\widehat{D}$. By analytic continuation, we first see that $u^{\infty}(\hat{x},k,d)=\widehat{u}^{\infty}(\hat{x},k,d)$ for
$(\hat x, k)\in \mathbb{S}^{n-1}\times\mathbb{R}_+$.
By Rellich's lemma (cf. \cite{CK}), from the assumption $u^{\infty}(\hat{x},k)=\widehat{u}^{\infty}(\hat{x},k)$ for all $\hat{x}\in \mathbb{S}^{n-1}$ it can be concluded that the
total waves fields $u(x,k)=\widehat{u}(x,k)$ for all $x\in \mathbb{G}$.
In particular,  we have
\be\label{uequhatu}
u(x,k)=\widehat{u}(x,k)\quad \mbox{for}\ x\in \pa B_R, \, k\in (0,\infty).
\en

{\bf Sound-soft Case.}
Consider first the case of sound-soft obstacles $D$ and $\widehat{D}$.
For the two dimensional case, we define
\ben
\wi{u}_2:=(\wi{\mathcal{K}}+L+\wi{\mathcal{S}}\circ W)\circ A(x\cdot d)\quad {\rm in}\,\,\R^2\ba\pa D.
\enn
Then, it is readily verified that $\wi{u}_2$ uniquely solves the following exterior Dirichlet boundary value problem
\ben
&&\Delta \wi{u}_2=0\quad {\rm in}\ \ \R^2\ba\ov{D},\\
&&\wi{u}_2=x\cdot d\quad {\rm on}\ \ \pa D,\\
&&\wi{u}_2(x)=O(1)\quad {\rm uniformly\, as}\ \ |x|\rightarrow\infty.
\enn
Similarly, one can define $\wi{\widehat{u}}_2$ associated to $\widehat{D}$. From \eqref{uasyD2D} and \eqref{uequhatu}, by comparing the coefficient of the term $k$, we found that
$\wi{u}_2=\wi{\widehat{u}}_2$ on $\pa B_R$. Note that both $\wi{u}_2$ and $\wi{\widehat{u}}_2$ are harmonic functions in $\R^2\ba\ov{B_R}$ and bounded at infinity. Thus by the
uniqueness of the exterior Dirichlet problem for Laplace's equation \cite{Kress}, we conclude that $\wi{u}_2=\wi{\widehat{u}}_2$ in $\R^2\ba\ov{B_R}$.
This further implies that $\wi{u}_2 = \wi{\widehat{u}}_2$ in $\mathbb{G}$ by the analytic extension.
Denote by $\wi{v}_2:=\wi{u}_2-x\cdot d$ and $\wi{\widehat{v}}_2:=\wi{\widehat{u}}_2-x\cdot d$ in $\R^2\ba\ov{D}$ and $\R^2\ba\ov{\widehat{D}}$, respectively.
Then $\wi{v}_2$ is also harmonic in $\R^2\ba\ov{D}$ and vanishing on $\pa D$.
Similarly, $\wi{\widehat{v}}_2$ is harmonic in $\R^2\ba\ov{\widehat{D}}$ and vanishing on $\pa \widehat{D}$.
Moreover, $\wi{v}_2 = \wi{\widehat{v}}_2$ in $\mathbb{G}$.
Since $D\neq\widehat{D}$, we have $\wi{v}_2$ is a harmonic function in $D^{\ast}$
with the homogeneous Dirichlet boundary $\wi{v}_2=0$ on $\pa D^{\ast}$. Using the maximum-minimum principle in $D^{\ast}$ and further the analytic extension in $\R^2\ba\ov{D}$, we conclude that
$\wi{v}_2=0$ in $\R^2\ba\ov{D}$. This readily implies $\wi{u}_2=x\cdot d$ in $\R^2\ba\ov{D}$.
However, this leads to a contradiction since, for $|x|\rightarrow\infty$, $\wi{u}_2(x)=O(1)$ uniformly in $\hat x$.

For the scattering problem in three dimensions, we define
\ben
\wi{u}_3:=\wi{\mathcal{K}}\circ A(1)\quad {\rm in }\ \ \R^3\ba\ov{D}.
\enn
Then, it is readily verified that $\wi{u}_3$ uniquely solves the following exterior Dirichlet problem
\ben
&&\Delta \wi{u}_3=0\quad {\rm in}\ \ \R^3\ba\ov{D},\cr
&&\wi{u}_3=1\quad {\rm on}\ \ \pa D,\cr
&&\wi{u}_3(x)=o(1)\quad {\rm uniformly\, as}\ \ |x|\rightarrow\infty.
\enn
Similarly, one can define $\wi{\widehat{u}}_3$ associated to $\widehat{D}$.
From \eqref{uequhatu} and \eqref{uasyD3D}, by comparing the coefficient of the term $k^{0}$, we found that $\wi{u}_3=\wi{\widehat{u}}_3$ on $\pa B_R$.
By uniqueness of the exterior Dirichlet problem for Laplace's equation \cite{Kress}, we conclude that $\wi{u}_3=\wi{\widehat{u}}_3$ in $\R^3\ba\ov{B_R}$.
This further implies that $\wi{u}_3 = \wi{\widehat{u}}_3$ in $\mathbb{G}$ by the analytic extension.
Since $D\neq\widehat{D}$, we deduce that $\wi{u}_3$ is a harmonic function in $D^{\ast}$
with Dirichlet boundary $\wi{u}_3=1$ on $\pa D^{\ast}$. Here, $D^*$ is the subdomain introduced earlier when discussing the admissible sound-hard obstacles. Using the maximum-minimum principle in $D^{\ast}$ and further the analytic extension in $\R^3\ba\ov{D}$, we conclude that
$\wi{u}_3=1$ in $\R^3\ba\ov{D}$. This leads to contradiction since, for $|x|\rightarrow\infty$, $\wi{u}_3(x)=o(1)$ uniformly in $\hat x$.

{\bf Sound-hard Case.}
We now turn to the case of sound-hard obstacles $D$ and $\widehat{D}$. Introduce the function
\ben
\wi{u}_N:=\wi{\mathcal{S}}\circ B(d\cdot\nu)\quad {\rm in}\,\R^n\ba\ov{D}.
\enn
Then, it is verified that $\wi{u}_N$ uniquely solves the following exterior Neumann problem
\ben
&&\Delta\wi{u}_N=0\;\;{\rm in}\, \R^n\ba\ov{D},\\
&&\frac{\pa\wi{u}_N}{\pa\nu}=d\cdot\nu\ \ {\rm on}\ \pa D,\\
&&\wi{u}_N(x)=o(1)\ \ {\rm uniformly\,as}\ |x|\rightarrow\infty.
\enn
Similarly, we introduce the function $\wi{\widehat{u}}_N$ associated to $\widehat{D}$.
From \eqref{uequhatu} and \eqref{uasyN2D}-\eqref{uasyN3D}, we deduce that
\ben
\wi{u}_N = \wi{\widehat{u}}_N \quad {\rm on}\,\pa B_R.
\enn
From \eqref{uequhatu}, \eqref{uasyN2D} in $\R^2$ and \eqref{uasyN3D} in $\R^3$, by comparing the coefficient of the term $k$,
we conclude that $\wi{u}_N = \wi{\widehat{u}}_N$ in $\mathbb{G}$.
Since $D\neq\widehat{D}$, we deduce that $\mathbbm{w}(x):=\wi{u}_N(x)-d\cdot x$ is a harmonic function in $D^{\ast}$ with the homogeneous Neumann boundary
$\frac{\pa \mathbbm{w}}{\pa\nu}=0$ on $\pa D^{\ast}$.
Since both $D$ and $\widehat{D}$ are admissible, we may apply the divergence theorem in $D^*$ to have
\ben
\int_{D^{\ast}}|\nabla \mathbbm{w}|^2dx = \int_{\pa D^{\ast}}\frac{\pa \mathbbm{w}}{\pa\nu}\ov{\mathbbm{w}}ds=0,
\enn
which further implies that $\mathbbm{w}=c$ in $D^{\ast}$ for some constant $c\in\C$. Again by the analytic continuation, we conclude that
$\mathbbm{w}=c$ in $\R^n\ba\ov{D}$; that is, $\wi{u}_N(x)=c+d\cdot x, x\in\R^n\ba\ov{D}$.
However, this is a contradiction, since for $|x|\rightarrow\infty$, one has that $\wi{u}_N(x)=o(1)$ uniformly in $\hat x$.

The proof is complete.
\end{proof}

\subsection{Recovery of the surrounding medium}

By Theorem~\ref{thm:main1}, we see that the embedded obstacle $D$ can be uniquely recovered, disregarding the surrounding medium $(\Omega\backslash\overline{D}, V)$. Now, we turn to the unique recovery of the medium parameter $V$.

\begin{theorem}\label{uni_q}
Let $(\Omega\backslash\overline{D}, V)$ and $({\Omega}\backslash\overline{D}, \widehat{V})$ be two mediums such that $u^{\infty}(\hat{x},k,d)=\widehat{u}^{\infty}(\hat{x},k, d)$
for $(\hat x, k)\in \Sigma\times \zeta$ and a fixed $d\in\mathbb{S}^{n-1}$, where $\Sigma\times\zeta$ is any open subset of $\mathbb{S}^{n-1}\times\mathbb{R}_+$.
Then $V=\widehat{V}$ under either one of the following admissibility conditions:
\begin{enumerate}
  \item[(i)] $D$ is sound soft and in two dimensional case $L\circ A(1)\neq 1$, both $V$ and $\widehat{V}$ are harmonic functions in $\Omega\ba\ov{D}$ satisfying the Dirichlet boundary conditions $V=\widehat{V}=0$ on $\pa D$;
  \item[(ii)] $D$ is sound hard and, both $V$ and $\widehat{V}$ are harmonic functions in $\Omega\ba\ov{D}$ satisfying the Neumann boundary conditions ${\pa V}/{\pa\nu}={\pa \widehat{V}}/{\pa\nu}=0$ on $\pa D$.
\end{enumerate}
\end{theorem}
\begin{proof}
By a same argument as that for the proof of Theorem~\ref{thm:main1}, one can show that the total fields coincide in $\R^3\ba B_R$, and furthermore their Cauchy date coincide on the boundary $\pa B_R$, i.e.,
\be\label{uequhatu2}
u(x,k)=\widehat{u}(x,k)\ {\rm and}\ \frac{\pa u(x,k)}{\pa\nu(x)}=\frac{\pa\widehat{u}(x,k)}{\pa\nu(x)},\ x\in \pa \Omega, \, k\in (0,\infty).
\en

{\bf Sound-soft $D$}. Let us first consider the case that $D$ is a sound-soft obstacle.
We introduce the function $\wi{u}_D$ in $\Omega\ba\ov{D}$ as follows,
\ben
\wi{u}_D:=\left\{
            \begin{array}{ll}
              \mathcal {F}\circ \wi{G}_V\circ\mathcal {F}(1), & \hbox{for 3D case;}\medskip \\
              \mathcal {F}\circ\wi{G_V}\circ\mathcal {F}(1)+\mathcal {F}\circ(1+c_2)U_V\circ\mathcal {F}(1)\\
              +\mathcal {F}\circ U_V\circ\mathcal {F}\circ (\wi{S}+c_2L)\circ A(1), & \hbox{for 2D case.}
            \end{array}
          \right.
\enn
Then one can verify that $\wi{u}_D$ is a solution to the following Dirichlet boundary value problem
\ben
\Delta \wi{u}_D=-V\mathcal {F}(1)\quad {\rm in}\ \ \Omega\ba\ov{D},\qquad \wi{u}_D=0\quad {\rm on}\ \ \pa D.
\enn
Similarly, we also introduce the corresponding function $\wi{\widehat{u}}_D$ associated to $\widehat{V}$.
From \eqref{uequhatu2} and \eqref{uasyD2D}-\eqref{uasyD3D}, by comparing the term of order $k^2$, one immediately has
\be\label{wequwhatD}
\wi{u}_D=\wi{u}_D\quad{\rm and}\quad\frac{\pa \wi{u}_D}{\pa\nu}=\frac{\pa \wi{u}_D}{\pa\nu} \quad {\rm on}\ \ \pa \Omega.
\en
Letting $\xi_D:=\wi{u}_D-\wi{\widehat{u}}_D$, we have
\be\label{vD}
\Delta \xi_D=(\widehat{V}-V)\mathcal {F}(1)\quad {\rm in}\ \ \Omega\ba\ov{D},
\en
and
\begin{equation}\label{VD2}
\xi_D=0\quad {\rm on}\ \ \pa D \quad{\rm and}\quad \xi_D=\frac{\pa \xi_D}{\pa\nu}=0\quad {\rm on}\ \ \pa \Omega.
\end{equation}
By the assumption, the difference $V_D:=\widehat{V}-V$ is a solution of the following boundary value problem
\be\label{pD}
\Delta V_D=0 \quad {\rm in}\ \ \Omega\ba\ov{D}\quad{\rm and}\quad  V_D=0 \quad {\rm on}\ \ \pa D.
\en
Using Green's theorem in $\Omega\ba\ov{D}$, by \eqref{vD}, \eqref{VD2} and \eqref{pD}, we deduce that
\ben
&&\int_{\Omega\ba\ov{D}}\mathcal {F}(1)|V_D|^2\ dx\cr
&=&\int_{\Omega\ba\ov{D}}(\Delta \xi_D \ov{V_D}-\xi_D\Delta\ov{V_D})\ dx\cr
&=&\int_{\pa \Omega}\left(\frac{\pa \xi_D}{\pa\nu} \ov{V_D}-\xi_D\frac{\pa \ov{V_D}}{\pa\nu}\right)\ ds-\int_{\pa D}\left(\frac{\pa \xi_D}{\pa\nu} \ov{V_D}-\xi_D\frac{\pa \ov{V_D}}{\pa\nu}\right)\ ds\cr
&=&0.
\enn
This further implies that $V_D=0$ in $\Omega\ba\ov{D}$ if the function $\mathbbm{f}:=\mathcal {F}(1)$ is sign preserving in $\Omega\ba\ov{D}$. That is, in such a case, one has the unique recovery result $V=\widehat{V}$ in $\Omega\ba\ov{D}$.
Next, we show that the function $\mathbbm{f}$ is sign preserving in $\Omega\ba\ov{D}$.
In the two dimensional case, $\mathbbm{f}$ is harmonic in $\R^2\ba\ov{D}$, vanishing on $\pa D$ and $\mathbbm{f}(x)\rightarrow 1-L\circ A(1)$ as $|x|\rightarrow\infty$.
Since $1-L\circ A(1)$ is a constant, by the maximum-minimum principle, the values of $\mathbbm{f}$ in $\R^2\ba\ov{D}$ are between $0$ and $1-L\circ A(1)$.
Thus, $\mathbbm{f}$ is sign preserving in $\R^2\ba\ov{D}$.
In three dimensions, $\mathbbm{f}$ is harmonic in $\R^3\ba\ov{D}$, vanishing on $\pa D$ and $\mathbbm{f}(x)\rightarrow 1$ as $|x|\rightarrow\infty$. Using again the
maximum-minimum principle, we deduce that $\mathbbm{f}$ is always positive in $\R^3\ba\ov{D}$. This completes the proof of the unique recovery of the medium $(\Omega\backslash\overline{D}, V)$ in the case that $D$ is a sound-soft obstacle.

{\bf Sound-hard $D$}. Consider now the case that $D$ is a sound-hard obstable.
We introduce the function $\wi{u}_N$ in $\Omega\ba\ov{D}$ as follows,
\ben
\wi{u}_N:=\left\{
            \begin{array}{ll}
              \wi{\mathcal{G}}_{V}(1)+\wi{\mathcal{S}}\circ B\circ\frac{\pa \wi{G}_V(1)}{\pa\nu}, & \hbox{for 3D case;}\medskip \\
              \wi{\mathcal{G}}_{V}(1)+\wi{\mathcal{S}}\circ B\circ\frac{\pa \wi{G}_V(1)}{\pa\nu}+c_2U_V(1), & \hbox{for 2D case.}
            \end{array}
          \right.
\enn
Then $\wi{u}_N$ is a solution of the following Neumann boundary value problem
\ben
\Delta \wi{u}_N=-V\quad {\rm in}\ \ \Omega\ba\ov{D},\qquad \frac{\pa \wi{u}_N}{\pa\nu}=0\quad {\rm on}\ \ \pa D.
\enn
Similarly, we introduce the corresponding function $\wi{\widehat{u}}_N$ associated to $\widehat{V}$.
From \eqref{uasyN2D}-\eqref{uasyN3D} and \eqref{uequhatu2}, by comparing the term of order $k^2$, one immediately has
\be\label{wequwhatN}
\wi{u}_N=\wi{u}_N\quad{\rm and}\quad\frac{\pa \wi{u}_N}{\pa\nu}=\frac{\pa \wi{u}_N}{\pa\nu} \quad {\rm on}\ \ \pa \Omega.
\en
Letting $\zeta_N:=\wi{u}_N-\wi{\widehat{u}}_N$, we have
\be\label{vN}
\Delta \zeta_N=\widehat{V}-V\quad {\rm in}\ \ \Omega\ba\ov{D},
\en
and
\begin{equation}\label{vN2}
\frac{\pa \zeta_N}{\pa\nu}=0\quad {\rm on}\ \ \pa D\quad{\rm and}\quad \zeta_N=\frac{\pa \zeta_N}{\pa\nu}=0\quad {\rm on}\ \ \partial \Omega.
\end{equation}
By the assumption, the difference $V_N:=\widehat{V}-V$ is a solution of the following boundary value problem
\be\label{pN}
\Delta V_N=0 \quad {\rm in}\ \ \Omega\ba\ov{D}\quad{\rm and}\quad  \frac{\pa V_N}{\pa\nu}=0 \quad {\rm on}\ \ \pa D.
\en
Using Green's theorem in $\Omega\ba\ov{D}$, by \eqref{vN}, \eqref{vN2} and \eqref{pN}, we deduce that
\ben
&&\int_{\Omega\ba\ov{D}}|V_N|^2\ dx\cr
&=&\int_{\Omega\ba\ov{D}}(\Delta \zeta_N \ov{V_N}-\zeta_N\Delta\ov{V_N})\ dx\cr
&=&\int_{\pa \Omega}\left(\frac{\pa \zeta_N}{\pa\nu} \ov{V_N}-\zeta_N\frac{\pa \ov{V_N}}{\pa\nu}\right)\ ds-\int_{\pa D}\left(\frac{\pa \zeta_N}{\pa\nu} \ov{V_N}-\zeta_N\frac{\pa \ov{V_N}}{\pa\nu}\right)\ ds\cr
&=&0.
\enn
This readily implies that $V_N=0$ in $\Omega\ba\ov{D}$; that is, $V=\widehat{V}$ in $\Omega\ba\ov{D}$.

The proof is complete.
\end{proof}

In Theorem~\ref{uni_q}, there are two admissibility conditions on the inhomogeneous medium $(\Omega\backslash\overline{D},V)$ under which it can be uniquely identified.
As an illustrative example, let us consider the case that $D$ is a polyhedron in $\mathbb{R}^n$ such that $\partial D$ is consisting of finitely
many cells $C_j$, $j=1,2,\ldots,\alpha$. Suppose that each cell has the following parametric representation
\[
\langle x,  \nu_j \rangle=l_j, \quad x\in C_j, \ 1\leq j\leq \alpha,
\]
where $\nu_j$ is the unit normal vector to $C_j$ and $l_j$ is the distance between the cell $C_j$ and the origin .
If $D$ is sound-soft, we let $(\Omega\backslash\overline{D}, V)$ be such that $V$ is a piecewise polynomial function associated to a certain polyhedral triangulation which as
a whole is an $H^2$-function. Moreover, it is assumed that for the piece touching the cell $C_j$, the parametric form of the polynomial is given by $x\cdot \theta_j+\tau_j$,
where $\theta_j\in\mathbb{R}^n$, $\tau_j\in\mathbb{R}$ and $\theta_j/\|\theta_j\|=\nu_j$, $\tau_j=-\|\theta_j\|\cdot l_j$.
By properly choosing the polynomials in the rest of the pieces, such a medium parameter function $V$ is harmonic (in the weak $H^2$ sense) in
$\Omega\backslash \overline{D}$ and satisfies the homogeneous Dirichlet boundary condition on $\partial D$.
Hence, by Theorem~\ref{uni_q}, both $D$ and $(\Omega\backslash\overline{D}, V)$ can be uniquely recovered.
Next, if $D$ is sound-hard, we let $V$ be a piecewise function associated to a certain polyhedral triangulation such that at each piece, it is a polynomial function,
and as a whole it is an $H^2$-function. In the piece touching the cell $C_j$, we assume that $V$ is of the form, $x\cdot \theta_j+\tau_j$,
where $\theta_j\in\mathbb{R}^n$, $\tau_j\in\mathbb{R}$ and $\theta_j\cdot \nu_j=0$.
By properly choosing the polynomials in the rest of the polyhedral pieces, we can also make such a medium parameter function $V$ harmonic (in the weak $H^2$ sense)
in $\Omega\backslash \overline{D}$ and satisfy the homogeneous Neumann boundary condition on $\partial D$.
Hence, by Theorem~\ref{uni_q}, both $D$ and $(\Omega\backslash\overline{D}, V)$ can be uniquely recovered.
Those remarks would find important applications if one intends to design a numerical recovery scheme of general $D$ and $(\Omega\backslash\overline{D}, V)$ by the
so-called {\it finite element method}, where one can approximate $D$ by a polyhedron and $V$ by a piecewise polynomial function.

\section{Numerics and discussions}

From the theoretical analyses given in the previous sections, by letting $k\rightarrow +0$, the buried obstacles produce more contribution to the scattered field, and thus also to the
far-field measurements. In this sense, the contribution from the surrounding inhomogeneous medium can be regarded as noise to the far-field measurements.
It is natural to reconstruct the buried obstacle by using the far-field measurements corresponding to low frequencies, while determining the surrounding medium by the far-field measurements corresponding to regular frequencies.

In our numerical simulations, we used the boundary integral equation method to compute the far-field patterns $u^\infty(\hat{x}_l, k_m, d_n)$  with $\hat{x}_l = 2\pi l/L,\,l=1,2,\cdots,L$,
$0<k_1<k_2<\cdots<k_M$, $d_n = 2\pi n/N$ for $L$ equidistantly distributed observation directions and $N$ equidistantly distributed observation directions.
We further perturb
$u^\infty(\hat{x}_l, k_m, d_n)$ by random noise using
\ben
u_{\delta}^\infty(\hat{x}_l, k_m, d_n) = (1+\delta\frac{s_1+is_2}{\sqrt{s_1^2+s_2^2}})u^\infty(\hat{x}_l, k_m, d_n),
\enn
where $s_1$ and $_2$ are two random values in $(-1,1)$ and $\delta$ presents the relative error.

In \cite{Potthast2010}, Potthast proposed the Orthogonal Sampling method based on the following indicator
\ben
I_{Potthast1}(z,d_n)=\sum_{m=1}^{M}\Big|\sum_{l=1}^{L} u^{\infty}(\hat{x}_l,k_m,d_n)e^{ik\hat{x}_i\cdot z}\Big|^2, \quad z\in\R^n,
\enn
for some fixed incident direction $d_n$. Motivated by the study in \cite{Liu2016}, we also consider the following indicator
\ben
I_{Liu1}(z,d_n)=\Big|\sum_{m=1}^{M}\sum_{l=1}^{L}e^{-ikd_n\cdot z}u^{\infty}(\hat{x}_l,k_m,d_n)e^{ik\hat{x}_i\cdot z}\Big|^2, \quad z\in\R^n.
\enn
The numerical simulations in \cite{Potthast2010} have shown that the indicator $I_{MF}(z,d_n)$ can be used to find the rough locations of the underlying obstacles, but the resolution to
the shape reconstruction is not so good. To solve this problem, Potthast suggested in \cite{Potthast2010} to use the following indicator
\ben
I_{PotthastN}(z)=\sum_{n=1}^{N}\sum_{m=1}^{M}\Big|\sum_{l=1}^{L}u^{\infty}(\hat{x}_l,k_m,d_n)e^{ik\hat{x}_i\cdot z}\Big|^2, \quad z\in\R^n,
\enn
with $N$ incident directions $d_n, \, n=1,2,\cdots,N$. Similarly, we also consider the following indicator
\ben
I_{LiuN}(z)=\Big|\sum_{n=1}^{N}\sum_{m=1}^{M}\sum_{l=1}^{L}e^{-ikd_n\cdot z}u^{\infty}(\hat{x}_l,k_m,d_n)e^{ik\hat{x}_i\cdot z}\Big|^2, \quad z\in\R^n.
\enn

In the following, we consider a benchmark example: the support of the surrounding inhomogeneous medium $\pa\Om$ is given by a round square, parameterized by
$x(t)= 2.25(\cos^3 t+\cos t,\,\sin^3 t+\sin t),\,0\leq t\leq2\pi$, whereas
the buried obstacle $D$ is given by a sound-soft "kite", parameterized by $x(t) =(\cos t+0.65\cos 2t-0.65, 1.5\sin t),\,0\leq t\leq2\pi$. Figure \ref{true1} shows the original
domain. The research domain is $[-6,6]\times[-6,6]$ with $121\times 121$ equally spaced sampling points.
We set the number of the observation directions $L=64$, the contrast function $q=0.5$.
The results by using a single incident direction are shown in Figure \ref{reconstruction1}. We observe that the indicators $I_{Potthast1}$ and $I_{Liu1}$ capture
the location of the buried "kite" by using the data corresponding to $10$ equally distributed frequencies in $[0.1, 2]$. To reconstruct the support of the surrounding inhomogeneous medium, we use $50$ equally distributed frequencies in $[0.1, 10]$.
The shape information can be improved by using more incident directions. Figure \ref{reconstruction2} shows the reconstructions with four
incident directions. In particular, $I_{LiuN}$ gives a rough shape reconstruction for buried "kite".
From Figures \ref{reconstruction1} and \ref{reconstruction2}, we found that our indicators $I_{Liu1}$ and $I_{LiuN}$ seemingly produce better reconstructions. We shall study the numerical
method in a forthcoming paper.
\begin{figure}[htbp]
  \centering
\includegraphics[width=3in]{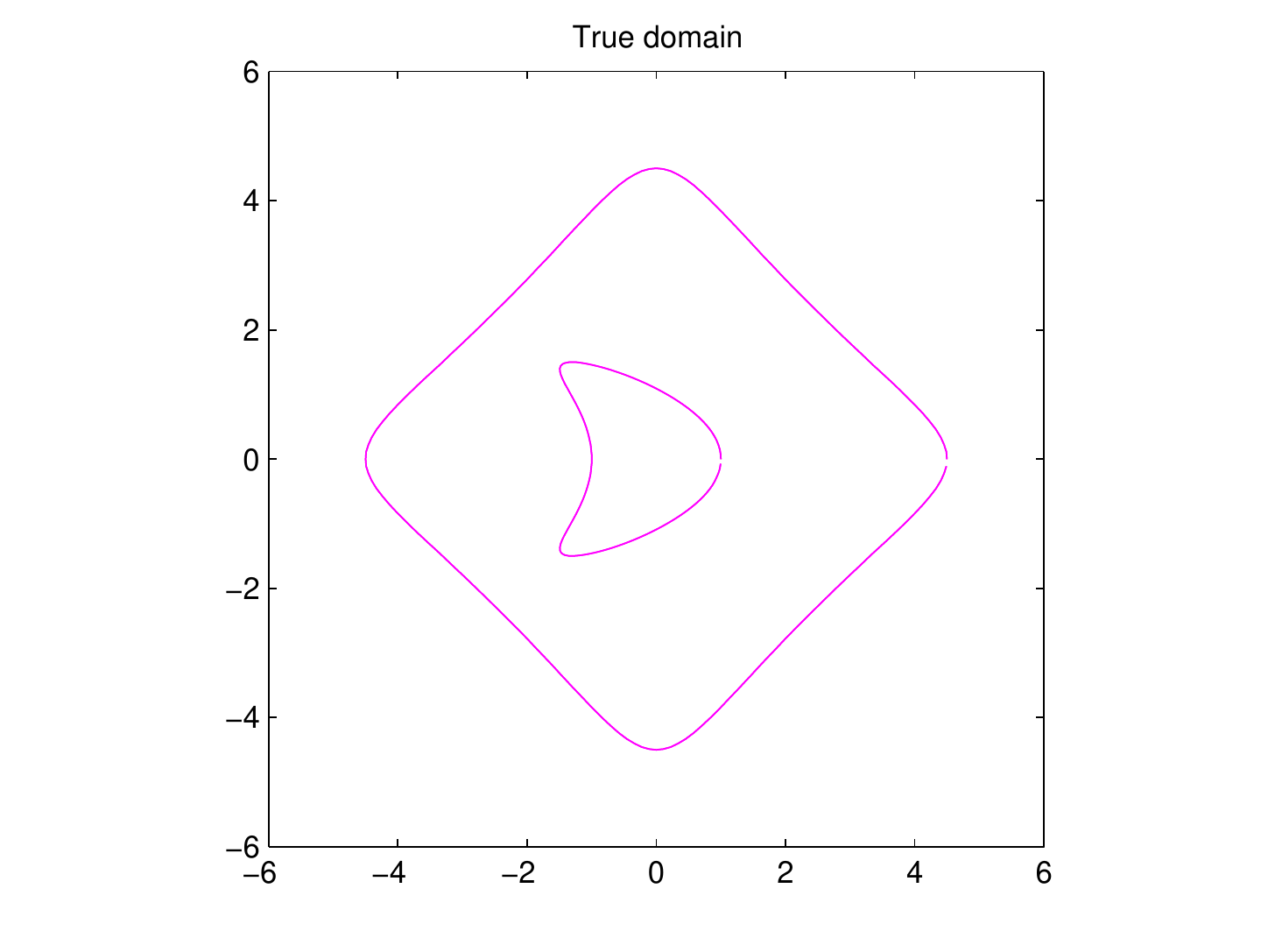}
\caption{True structure}\label{true1}
\end{figure}

\begin{figure}[htbp]
  \centering
\includegraphics[width=6in]{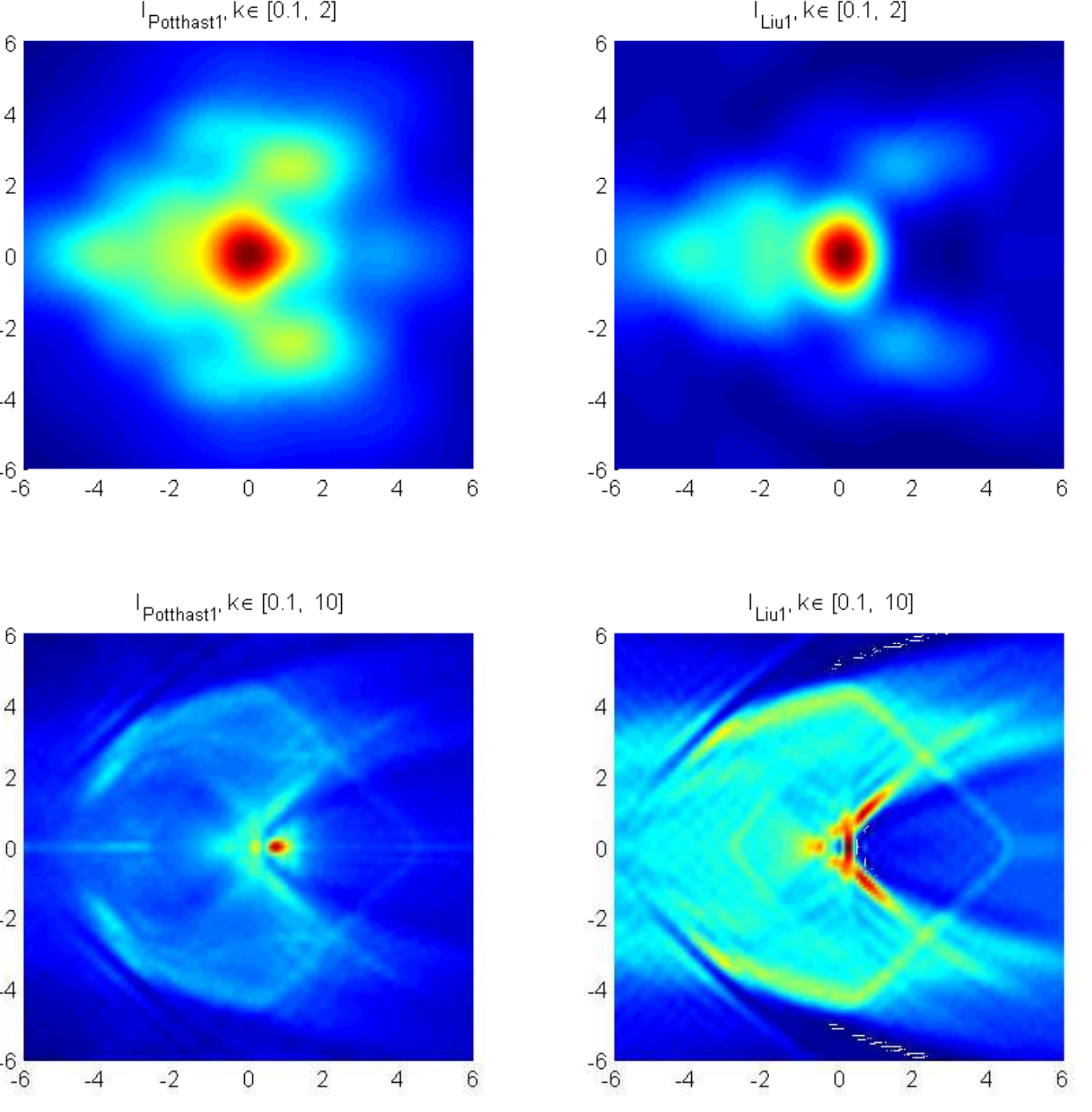}
\caption{Reconstruction with a single incident direction $(-1,0)$ and $10\%$ noise.}\label{reconstruction1}
\end{figure}

\begin{figure}[htbp]
  \centering
\includegraphics[width=6in]{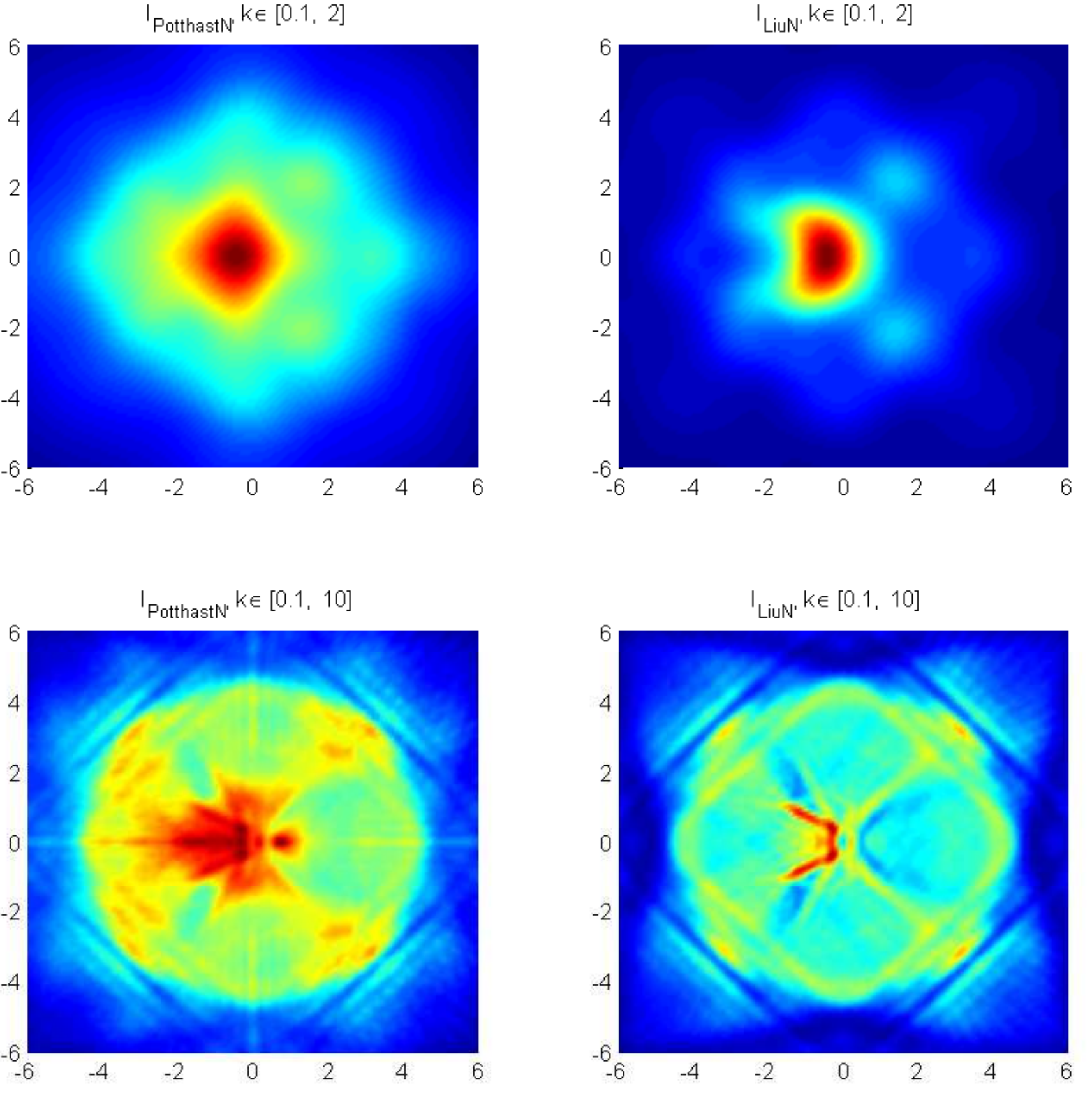}
\caption{Reconstruction with four incident directions $(1,0),(0,1),(0,-1),(-1,0)$ and $10\%$ noise.}\label{reconstruction2}
\end{figure}


\section*{Acknowledgement}

The work of H. Liu was supported by the FRG fund from Hong Kong Baptist University, the Hong Kong RGC grants (projects 12302415 and 405513) and NNSF of China (No.\,11371115).
The work of X. Liu was supported by the NNSF of China under grants 11571355, 61379093 and 91430102.

\end{document}